%% file: Article-new.tex
\documentclass{amsart}

\usepackage[T1]{fontenc}
\usepackage[utf8]{inputenc}

\usepackage{enumerate}

\usepackage{amssymb,amsmath,amsthm,amscd}
\usepackage{mathrsfs}
\usepackage[all,dvips]{xy}

\usepackage{float}
\floatstyle{boxed} 
\restylefloat{figure}

\usepackage{subfig}
\usepackage{pstricks-add}

\input{macrosmath} 

\theoremstyle{plain} 
\newtheorem{theo}{Theorem}[section]
\newtheorem{cor}[theo]{Corollary}
\newtheorem{prop}[theo]{Proposition}
\newtheorem{prop-defi}[theo]{Proposition-Definition}
\newtheorem{lem}[theo]{Lemma}

\newtheorem*{theoetoile}{Theorem} 

\newtheorem*{propetoile}{Proposition}
\newtheorem*{coretoile}{Corollary}

\newtheorem*{defitoile}{Definition}
\newtheorem{proper}[theo]{Property}

\theoremstyle{definition} 
\newtheorem{defi}[theo]{Definition}

\theoremstyle{remark}
\newtheorem{exemple}[theo]{Example}
\newtheorem{remarque}[theo]{Remark}

\title{The regular and profinite representations of residually finite groups}

\author{Jean-Fran\c{c}ois Planchat}
\address{Mathematisches Institut, Bunsenstr. 3-5, D-37073 G\"{o}ttingen Germany }
\email{planchat@uni-math.gwdg.de}
\thanks{This work is a part of my PhD thesis done at the Université Paris 7 Denis Diderot. I would like to thank my advisor Andrzej \.{Z}uk for his constant support and help, as well as Georges Skandalis for useful discussions.}


\begin{document}

\begin{abstract}
Let $G$ be a residually finite group. To any decreasing sequence $\mathcal S = \left( H_n\right)_n $ of finite index subgroups of $G$ is associated a unitary representation $\rho_{\mathcal S}$ of $G$ on the Hilbert space $\bigoplus_{n=0}^{+\infty} \ell^2 \left( G/H_n\right) $. This paper investigates the following question: when does the representation $\rho_{\mathcal S} $ weakly contain the regular representation $\lambda$ of $G$?  
\end{abstract}

\maketitle

\section*{Introduction}

Let $G$ be a countable residually finite group and $ H$ be a finite index, not necessarily normal, subgroup. The group $G$ acts on $G/H $ and we let $\lambda_{G/H}$ denote the representation that $G$ admits on $\ell^2 \left(G/H \right)$. That is, we let 
\[\lambda_{G/H}(g)(\xi)(x)=\xi(g^{-1}\cdot x)\]  
with $g\in G, \ \xi \in \ell^2 \left(G/H \right) $ and $x \in G/H$. In this paper, we will mainly be interested in representations of the form 
\[\rho_{\mathcal S} := \bigoplus_{n=0}^{+\infty} \lambda_{G/H_n}  
\]
where $\mathcal S = \left( H_n\right)_n $ is a decreasing sequence of finite index, not necessarily normal, subgroups of $G$. Such representations are called \emph{profinite representations}. We will investigate the link between such representations and the regular representation $\lambda$ of $G$. Recall that $\lambda$ is the representation that $G$ admits on $\ell^2 \left( G\right)$ and defined for all $g\in G, \ \xi \in \ell^2 \left(G \right) $ and $h \in G$ by
\[
 \lambda(g)(\xi)(h)=\xi(g^{-1}h).
\]

As $\rho_{\mathcal S}$ is a sum of \emph{finite representations} (i.e. factorizing through a finite index subgroup of $G$) and $\lambda$ is a $C^0$\emph{-representation} (i.e. the coefficients $\left\langle \lambda(g)(\xi)\ | \ \psi \right\rangle $ tend to 0 when $g$ tends to infinity), as soon as $G$ is infinite, none of these two representations can be a subrepresentation of the other \cite{Dix77}. However, one can expect a weaker link between them, namely weak containment. The goal of this work can be stated as follows: when does the representation $\rho_{\mathcal S } $ weakly contain the regular representation $\lambda$ of $G$? 

Let us recall the definition of weak containment. 

\begin{defi}\label{defiwc}
If $G$ is a countable group and $\left(\pi_1, \mathcal H _1\right)$, $\left(\pi_2,\mathcal H _2 \right)$ are two unitary representations of $G$, we say that  $\pi_2$ weakly contains  $\pi_1$ (and write $\pi_1 \prec \pi_2$ ) if and only if for all $\epsilon>0$, $\xi \in  \mathcal H _1$ and $ K \subset G$ finite, there are vectors $\nu_1,\nu_2,\dots,\nu_n$ in $\mathcal H _2$ such that
\begin{equation}\label{wc}
 \forall g \in K, \ \left| \left\langle \pi_1(g)\xi | \xi\right\rangle - \sum_{i=1 }^n \left\langle \pi_2 (g)\nu_i | \nu_i \right\rangle \right| < \epsilon.
\end{equation}
\end{defi}

There is an equivalent formulation of this property in the framework of $C^*$-algebras \cite{Dix77,Fell60}: for every unitary representation $\left(\pi,\mathcal H \right) $ of $G$, we extend $\pi$ linearly to get a $*$-homomorphism 
\[
 \ell^1 \left(G\right) \xrightarrow{\pi} \mathcal B \left( \mathcal H  \right)
\]
where $\mathcal B \left( \mathcal H  \right) $ denotes the algebra of bounded operators on $\mathcal H $. One then defines 
\[
 \cs{\pi}\left( G \right):= \overline{\pi \left( \ell^1 \left(G\right) \right)}
\]
where the closure is taken with respect to the $C^*$-operator norm  $\lVert \cdot \rVert_{_{\mathcal B \left( \mathcal H  \right) }}$ on $\mathcal B \left( \mathcal H \right) $. Then, 
\begin{eqnarray*}
 \pi_1 \prec \pi_2  & \Longleftrightarrow & \cs{\pi_2}\left( G \right)\twoheadrightarrow\cs{\pi_1}\left( G \right)  \\
                    & \Longleftrightarrow & \forall f \in \ell^1\left( G \right),\ \lVert \pi_1 (f) \rVert_{_{\mathcal B \left(\mathcal H _1 \right)}} \leq \lVert  \pi_2 (f) \rVert_{_{\mathcal B \left( \mathcal H _2 \right)}} \\
                    &  \Longleftrightarrow & \forall f \in \ell^1\left( G \right),\ \text{sp}\left(\pi_1 (f) \right) \subset \text{sp}\left(\pi_2 (f) \right)
\end{eqnarray*}
where $\text{sp}\left(M \right) $ denotes the spectrum of a operator $M$. In particular, one remarks that $\pi_1$ and $\pi_2$ are weakly isomorphic (written $\pi_1 \simeq \pi_2$ and defined to be $\pi_1 \prec \pi_2$ and $\pi_2 \prec \pi_1$) if and only if the $C^*$-algebras $\cs{\pi_1}\left( G \right)$ and $ \cs{\pi_2}\left( G \right) $ are isomorphic. \\

The representation $\rho_{\mathcal S}$ comes from an action of $G$ on a countable set. In the preliminary Section \ref{preliminary}, we introduce the notion of \emph{locally somewhere free} action of a countable group $G$ on a measure space $\left(X,\nu\right)$: 

\begin{defitoile}[\ref{deflsf}]
We say that the action of $G$ on $\left(X,\nu\right)$ is \emph{locally somewhere free} (l.s.f. for short) if for every finite subset $K \subset G$, there is a Borel subset $F \subset X$ of positive measure such that the non-trivial elements of $K$ fixe almost no point in $F$:
 \[
 \forall \gamma \neq 1 \in K,\  x\in F \Rightarrow \gamma \cdot x \neq x \text{ a.e.}
 \] 
\end{defitoile}

If $X$ is metrizable, locally compact and separable, and $\nu$ is a Radon measure quasi-preserved by $G$, then one can prove the following:

\begin{propetoile}[\ref{lemsuff}]
 If the action of $G$ on $X$ is l.s.f. then the regular representation $\lambda$ of $\Gamma$ is weakly contained in $\rho_X$.
\end{propetoile}

Here, $\rho_X$ denotes the canonical unitary representation that $G$ admits on $L^2\left(X\right)$.

In order to use this result, we remark in Section \ref{sec1} that there is a correspondence between profinite representations and actions on \emph{rooted trees}. More precisely, to any decreasing sequence $\mathcal S = \left( H_n\right)_n $ of finite index subgroups of $G$, is associated a rooted tree $\tree_{\mathcal S}$ together with a spherically transitive action of $G$ on it. This correspondence is one-to-one and the representation $\rho_{\mathcal S}$ is isomorphic to the permutational representation $\rho_{\tree_{\mathcal S}}$ coming from the action of $G$ on $\tree_{\mathcal S}$. \\

In the sequel, we mostly use this interpretation to adress our problem. We thus briefly recall in Section \ref{sec2} some well-known facts about the automorphism group $\aut$ of a rooted tree. It turns out that the property $\lambda \prec \rho_{\tree}$ is linked to the size of the stabilizers $\St_G (v\tree) $ in $G$ of the \emph{subtrees} $v\tree$ of $\tree$. In Section \ref{sec3}, we use Proposition \ref{lemsuff} of Section \ref{sec1} in order to prove the following:

\begin{theoetoile}[\ref{theosuffisant}] Let $G$ be a countable group acting faithfully on a rooted tree $\tree$. 
\begin{enumerate}[i.]
 \item If for every vertex $v\in \tree^0$, the stabilizer $\St_G (v\tree)$ is trivial, then $\lambda \prec \rot$.
 \item More generally, if the set $\bigcup_{v\in \tree^0} \St_G (v\tree)$ is finite and has cardinality $n$, then $\lambda \prec \rot^{\otimes n}$. 
 \item One always has  $\lambda \prec \bigoplus_{n=1}^{+\infty}\rot^{\otimes n} $.
\end{enumerate}
\end{theoetoile}

The proof of the first part amounts to show that the triviality of the stabilizers $\St_G (v\tree)$ is equivalent to the action of $G$ on $\tree$ to be l.s.f., so that Proposition \ref{lemsuff} can be applied. In Section \ref{sec4}, we study the inverse implication of this part of the Theorem.  
 
\begin{theoetoile}[\ref{theonec}]
Let $G$ be a countable group in which the normalizer $N_G \left(H \right)$ of any non-central finite group $H$ has infinite index in $G$.

 Suppose that $G$ acts spherically transitively on a rooted tree $\tree$. If there exists a subtree $v\tree$ whose stabilizer $\St_G \left( v\tree\right)$ in $G$ is not trivial, then the $*$-homomorphism $\rho_\tree$ defined on $\C G$ is not faithful. In particular, $\lambda \nprec \rho_{\mathcal S}$.  
\end{theoetoile}

Here, we have to make an algebraic assumption on $G$. Indeed, the sufficient condition in Theorem \ref{theosuffisant}.\ref{theosuffisant1} is not necessary in general (see Example \ref{exnec}).\\

In the last section, we illustrate the results of the previous two. In particular, we show that for the following classes of group, any faithful and spherically transitive action on a rooted tree is l.s.f.:

\begin{enumerate} \label{ZOB}
 \item \label{ZOB1} torsion free Gromov hyperbolic groups,
 \item \label{ZOB2} uniform lattices in a connected simple real Lie group $G$ with finite center and $\R$-rank 1,
 \item \label{ZOB3} irreducible lattices in a connected semisimple real Lie group with finite centre, no compact factor and $\R$-rank $\geq 2$.
\end{enumerate}

Thus, if $G$ belongs to one of these classes, then any faithful profinite representation of $G$ weakly contains the regular. Here is an application:

\begin{coretoile}[\ref{corHTS}]
Let $\Gamma$ be non-elementary, torsion free, residually finite hyperbolic group.  Let $\mathcal S = \left( H_n\right)_n $ be a decreasing sequence of finite index subgroups of $\Gamma$ such that the representation $\ros$ is faithful. 

Let $M_F := \frac{1}{2|F|}\sum_{g\in F} g+g^{-1}$ be the Markov operator associated to a finite generating set $F$ of $\Gamma$ which does not contain 1. Then 
\[
 \left[ -\frac{1}{|F|}, \frac{\sqrt{2|F|-1}}{|F|}\right] \subset \overline{\bigcup_{n} \text{sp} \left( \lambda_{\Gamma/H_n}\left(M_F \right) \right)} .
\]
\end{coretoile}

Finally, we study the case of \emph{weakly branched} subgroups of the automorphism group $\aut$ of a \emph{regular} rooted tree $\tree$. They form an interesting and wide class of residually finite groups which provides important examples (infinite, periodic and finitely generated groups \cite{aleshin72}, finitely generated groups with intermediate growth \cite{grig84}, amenable groups not belonging to the class SG \cite{grigzukbasilica,moybasilica}, finitely generated groups with non-uniform exponential growth \cite{wilson,wilsonbis,bartnonunif}).

A direct consequence of their definition is that their action on $\tree$ is not l.s.f. Thus, the results of \S \ref{subsecLattices} imply that a weakly branched group cannot belong to one the classes (\ref{ZOB1}),(\ref{ZOB2}) or (\ref{ZOB3}) above. It seems to the author that this was not known. 

Finally, we prove the following proposition, which in particular indicates that Theorem \ref{theosuffisant}.\ref{theosuffisant3} is optimal:

\begin{propetoile}[\ref{propwb}]
 Let $G$ be a weakly branched subgroup of $\aut$. For every $n>0$, the $*$-homomorphism $\rot^{\otimes n}$ defined on $\C G$ is not faithful. In particular, the representation $\rot^{\otimes n}$ does not weakly contain the regular $\lambda$.
\end{propetoile}

\section{Locally somewhere free actions and weak containment of the regular representation}\label{preliminary}

Let $G$ be a countable group. Let $\left(X,\nu\right)$ be a metrizable, locally compact and separable space endowed with a Radon measure $\nu$, on which $G$ acts. If the action is measurable and quasi-preserves $\nu$, then it yields the following unitary representation:
\[
\begin{array}{ccc}
G  & \overset{\rho_X}{\longrightarrow}  & \mathcal U \left( L^2 \left(X\right)\right)   \\
 g  &  \longrightarrow   &  f \to \rho_X\left(g\right)\left(f\right) 
\end{array}
\]
where $\rho_X\left(g\right)\left(f\right) $ is defined by 

\[\rho_X\left(g\right)\left(f\right)\left(x\right)= \left(\frac{dg^{-1}\nu}{d\nu}\left(x\right)\right)^{1/2} f\left(g^{-1}\cdot x \right) \]
with $\frac{dg^{-1}\nu}{d\nu} $ the Radon-Nikodym derivative. 

\begin{defi}\label{deflsf}
We say that the action of $G$ on $\left(X,\nu\right)$ is \emph{locally somewhere free} (l.s.f. for short) if for every finite subset $K \subset G$, there is a Borel subset $F \subset X$ of positive measure such that the non-trivial elements of $K$ fixe almost no point in $F$:
 \[
 \forall \gamma \neq 1 \in K,\  x\in F \Rightarrow \gamma \cdot x \neq x \text{ a.e.}
 \] 
\end{defi}

\begin{prop}\label{lemsuff}
 If the action of $G$ on $X$ is l.s.f. then the regular representation $\lambda$ of $\Gamma$ is weakly contained in $\rho_X$.
\end{prop}

\begin{proof}
 
Let $K$ be a finite subset of $X$ and $F$ a subset of $X$ of positive measure such that the non-trivial elements of $K$ fixe almost no point in $F$. Let us consider a distance $d$ on $X$ compatible with its topology and for each positive integer $n$, the following measurable set: 
 \[
  E_n := \left\lbrace x \in F \ | \ \forall \gamma \neq 1 \in  K, \ d(\gamma \cdot x, x) > 1/n\right\rbrace. 
 \]

By definition, $\nu\left(F \setminus \cup_n E_n \right)= 0$ and therefore, there is a positive integer $k$ such that $\nu \left( E_k\right)>0$.  Since $X$ is separable, there is ball $B\left(x,l\right)$ of radius $l<\frac{1}{4k}$ such that $U_K:=B\left(x,l\right) \cap E_k$ has non-zero measure. It satisfies
\begin{equation}\label{eqfree}
 \forall \gamma \neq 1 \text{ in } K, \ \gamma \left( U_K\right) \cap  U_K = \emptyset.
\end{equation}

Let us now prove that $\lambda \prec \rho_X$. We denote by $\delta_{1}$ the Dirac function over the identity element of $G$. This is a cyclic vector for the regular representation $\lambda$, i.e. the family $\left\lbrace \lambda\left(g\right)\left( \delta_{1}\right)= \delta_{g} \ | \ g\in G \right\rbrace $ is total in $\ell^2 \left(\Gamma \right)$. Therefore, by a result of Fell \cite{Fell63}, it is enough to check (\ref{wc}) in Definition \ref{defiwc} for $\xi=\delta_{1}$. As such, we consider $K$ a finite subset of $\Gamma$ and $U:=U_K$ the measurable subset associated to $K$ defined previously. If we let $\chi_{U} \in L^2\left(X\right) $ be the characteristic function of $U$, then (\ref{eqfree}) implies that for $g\in K$ 
\[
 \left< \lambda\left(g\right) \left(\delta_{1}\right) \ | \ \delta_1\right> = \delta_{1,g} =\left< \rho_X\left( g\right)\left(\frac{\chi_{U} }{ \nu\left(U\right)} \right)\ | \  \frac{\chi_{U} }{ \nu\left(U\right)}   \right> 
\]
where $\delta_{1,g}=1$ if $1=g$ and $0$ otherwise.
\end{proof}

\begin{exemple}
 The action of a non-elementary, torsion free Gromov hyperbolic group on its boundary $\left(\partial G,\nu\right)$ endowed with a Patterson-Sullivan measure $\nu$ fullfills all the conditions of the previous proposition (see \cite{coo93}). It is l.s.f. since each element in $G$ admits exactly two fixed points in $\partial G$, and the measure $\nu$ has no atom.   
\end{exemple}

We stress that a faithful action is not necessarily l.s.f.; the last chapter will give such examples.

\section{Rooted trees}\label{sec1}

Let $\bar{d}=d_0,d_1,\dots,d_n,\dots$ be a sequence of integers with $d_i \geq2$ for all $i$. We define the rooted tree $\tree_{\bar{d}}$ as follows: $\tree_{\bar{d}}$ is an infinite, locally finite tree endowed with the usual metric $dist$ carried by any graph, and for which:
\begin{enumerate}[1)]
 \item there is a particular vertex $\varnothing$ called the \emph{root},
 \item the degree $\textrm{deg}(v)$ of any vertex $v$ depends only on its distance to the root and is more precisely given by 

\[\textrm{deg}(\varnothing)=d_0, \text{ and for } n=dist(\varnothing,v)\geq 1, \  \textrm{deg}(v)=d_n +1. \]
\end{enumerate}

Let $\tree_{\bar{d}}^0$ denotes the set vertices of $\tree_{\bar{d}}$. If $n$ is a non-negative integer, the set of vertices whose distance to the root equals $n$ is called the $n$-th level of $\tree_{\bar{d}}$ and is denoted by $L_n$. One can give $\tree_{\bar{d}}$ a \emph{planar graph} structure as follows: for every $n$, one labels each vertex of $L_n$ by a finite sequence $i_1 i_2 \dots i_n$ with $i_j \in \left\lbrace 1,2,\dots,d_j \right\rbrace $ such that 2 vertices  $i_1 i_2 \dots i_n \in L_n$ and $i'_1 i'_2 \dots i'_{n+1} \in L_{n+1}$ are connected by an edge if and only if $i_1 i_2 \dots i_n = i'_1 i'_2 \dots i'_n $. Every level $L_n$ is then endowed with the lexicographic order.  

Given any vertex $v$, one defines the \emph{subtree} $v\tree_{\bar{d}}$ to be the connected subgraph of $\tree_{\bar{d}}$ whose vertices $w$ are \emph{descendants} of $v$ (i.e. $v \in \left[\varnothing,w \right] $ where $\left[\varnothing,w \right]$ denotes the unique geodesic path linking $\varnothing$ and $w$). 

We refer to Fig. \ref{figarbre} for a less rigourous but more visual presentation of these definitions.

\begin{figure}[ht]
\begin{center}
\begin{pspicture}(6,7)
\put(3.2,6){$\mbox{\small $\varnothing$}$}

\psline(3,6)(1.5,5) \psline(3,6)(2,5) \psline[linestyle=dotted](2.3,5)(2.7,5) \psline[linestyle=dotted](3.2,5)(4.1,5) \psline(3,6)(4.5,5) 
\put(1.38,5){$\mbox{\tiny $1$}$} \put(1.88,5){$\mbox{\tiny $2$}$} \put(4.62,5){$\mbox{\tiny $d_0$}$}

\psline(1.5,5)(0.5,4) \psline(1.5,5)(0.8,4) \psline[linestyle=dotted](1.1,4)(1.3,4) \psline(1.5,5)(1.7,4) 
\put(.30,4){$\mbox{\tiny $11$}$} \put(1.8,4){$\mbox{\tiny $1 d_1$}$}
\psline(4.5,5)(4.3,4) \psline(4.5,5)(4.6,4) \psline[linestyle=dotted](4.9,4)(5.2,4) \psline(4.5,5)(5.5,4) 
\put(3.90,4){$\mbox{\tiny $d_01$}$} \put(5.5,4){$\mbox{\tiny $d_0 d_1$}$}

\psline[linestyle=dotted](0.5,4)(-0.3,3) \psline[linestyle=dotted](5.5,4)(6.3,3) 

\put(3.5,3){$\mbox{\tiny $v$}$}
\psline(3.5,3)(3,2) \psline(3.5,3)(3.2,2) \psline[linestyle=dotted](3.5,2)(3.7,2) \psline(3.5,3)(4,2)
\psline[linestyle=dotted](3,2)(2.5,1) \psline[linestyle=dotted](4,2)(4.5,1)
\psline[linestyle=dashed](3,6)(3,5) \psline[linestyle=dashed](3,5)(3.2,4) \psline[linestyle=dashed](3.2,4)(3.5,3)

\psframe[framearc=1,linestyle=dotted](0.9,4.85)(5.1,5.15) \put(5.5,5){$\mbox{\tiny $L_1$}$}
\psframe[framearc=1,linestyle=dotted](0,3.85)(6.2,4.15) \put(6.3,4){$\mbox{\tiny $L_2$}$}

\psparabola[linestyle=dotted](2,0.5)(3.5,3.2) \put(4.5,2.5){$\mbox{\tiny $v \tree_{\bar{d}} $}$}

\end{pspicture} \end{center}\caption{Rooted tree $\tree$}\label{figarbre}  \end{figure}

A countable group $G$ is said to act on a rooted tree $\tree_{\bar{d}}$ if it acts on the underlying graph while fixing the root. Such an action induces a unitary representation $\rho_{\tree_{\bar{d}}}$ of $G$ on the Hilbert space $\ell^2 \left( \tree_{\bar{d}}^0\right)$ defined by:
\[
\forall g \in G, \ \forall \xi \in \ell^2 \left( \tree_{\bar{d}}^0\right), \ \forall v \in \tree_{\bar{d}}^0, \ \rho_{\tree_{\bar{d}}}(g)(\xi)(v)=\xi(g^{-1}\cdot v).
\]

Let us also define, for every vertex $v \in \tree_{\bar{d}}^0$ and every non-negative integer $n$ the following stabilizers:
\[
 \St_G (v):=\left\lbrace g \in G \ | \ g(v)=v\right\rbrace, \quad  \St_G (L_n):=\left\lbrace g \in G \ | \ g(v)=v, \ \forall v \in L_n \right\rbrace
\]
which both have finite index in $G$, since $G$ preserves the finite sets $L_n$. The representation $\rho_{\tree_{\bar{d}}} $ will be faithful if and only if the action of $G$ on $\tree_{\bar{d}} $ is, that is if and only if $\bigcap_{n=0}^{\infty}\St _G (L_n)$ is trivial.\\

The action of $G$ on $\tree_{\bar{d}}$ is said \emph{spherically transitive} if $G$ acts transitively on each level. In this case, one has
\[
\forall n, \ \forall v \in L_n, \ \St_G (L_n)=\bigcap_{g\in G}g\St_G (v)g^{-1} .
\]

One notices that there is a correspondence between spherically transitive actions of $G$ on a rooted tree $\tree_{\bar{d}}$ and decreasing sequences $\mathcal S=\left( H_n\right)_n $ of finite index subgroups of $G$. Moreover, this correspondence respects the representations $\rho_{\tree_{\bar{d}}} $ and $\rho_{\mathcal S} $. 

Indeed, suppose $G$ acts spherically transitively on $\tree_{\bar{d}}$ and let $\left( v_n\right)_n$ be the sequence with $v_n$ the \emph{left most} vertex $11\dots1$ of $L_n$. Then, $\mathcal S\left(\tree_{\bar{d}} \right):=\left(\St_G (v_n)\right)_n$ is a decreasing sequence of finite index subgroups of $G$ fullfilling 
\[\forall n,\ \left|\St _G (v_{n})/\St _G (v_{n+1})  \right|=d_n .\] 
Moreover, since the actions of $G$ on $L_n$ and on $G/\St_G(v_n)$ are isomorphic, the representations  $\rho_{\tree_{\bar{d}}} $ and 
 $\rho_{\mathcal S (\tree_{\bar{d}})} $ are isomorphic.\\

Conversely, let $\mathcal S=\left( H_n\right)_n $ be a decreasing sequence of finite index subgroups of $G$. Without lost of generality, one can assume that $H_0=G$ and  $\mathcal S$ is strictly decreasing. Then, we can construct a rooted tree $\tree \left(\mathcal S \right)$ as follows: 
\begin{itemize}
 \item[-] its vertices are the cosets $g_{n,k}H_n$,
 \item[-] two vertices $g_{n,k}H_n $ and $g_{n',k'}H_{n'}$ are connected by an edge if and only if $n'=n+1$ and $g_{n',k'}H_{n'} \subset  g_{n,k}H_n$.
\end{itemize}

By construction, the root of $\tree \left(\mathcal S \right)$ is represented by $G$ and, more generally, the sets $L_n $ and $ G/H_n$ are equal:  $\tree \left(\mathcal S \right)$ is actually the rooted tree $\tree_{\bar{d}}$ where $\bar{d}=d_0,d_1,\dots$ is caracterised by $d_n=\left| H_n/H_{n+1} \right|$. Moreover, it follows that $G$ acts on $L_n$ for all $n$, and thus on $\tree \left( \mathcal S \right)$. The representations $\rho_{\mathcal S}$ and $\rho_{\tree \left( \mathcal S \right)}$ are clearly isomorphic.\\

This correspondence implies for instance that a countable group $G$ acts faithfully on a rooted tree if ond only if it is residually finite. We will soon see that this picture is relevant for our purpose. More precisely, the weak containment of the regular representation $\lambda$ in $\ros$, equivalently in $\rot$, is linked to the size of the stabilizers $\St_G \left( v\tree \right)$ of subtrees defined by  
\[
\forall v \in \tree^0, \  \St_G \left( v\tree \right) := \left\lbrace g\in G \ | \ g(x)=x, \ \forall x\in v\tree\right\rbrace .
\]

Before starting the study of our problem, we recall in next section some convinient facts about the group of the automorphisms of a rooted tree.

\section{The automorphism group of a rooted tree}\label{sec2}

By $\text{Aut}\left(\tree_{\bar{d}} \right) $ we denote the group of the automorphisms of $\tree_{\bar{d}}$ that fixe the root $\varnothing$. This short section aims to recall some convenient tools to describe such automorphisms. We refer to \cite{bourbakiandrzej,livrenek,nekcstar,mathese} for more details. \\

For all $n$, the group $\text{Aut}\left(\tree_{\bar{d}} \right) $ preserves the level $L_n$ as well as the finite set $\left\lbrace v\tree_{\bar{d}} \ | \ v \in L_n\right\rbrace $ of subtrees rooted at its vertices. 

Thanks to the self-similar structure of $\tree_{\bar{d}} $, the group $\text{Aut}\left(\tree_{\bar{d}} \right) $ admits a natural decomposition in terms of the automorphisms group of other rooted trees. 

More precisely, we have just noticed that every $g$ in $\text{Aut}\left(\tree_{\bar{d}} \right)$ induces a permutation $g_1$ on the set $L_1$, as well as an isomorphism $\varphi_{g(v)}(g)$ from $v\tree_{\bar{d}}$ onto $g(v)\tree_{\bar{d}}$, for every vertex $v$ in $L_1$. 

These two subtrees are canonically isomorphic to $\tree_{\sigma (\bar{d})}$ where $\sigma$ denotes the \emph{shift} on $\R^{\N}$ (i.e. $\sigma (\bar{d})=d_1,d_{2},\dots $) and $\varphi_{g(v)}(g)$ can be seen therefore as an element of $\text{Aut}\left(\tree_{\sigma (\bar{d})} \right) $. 

It is easy to see that this data completely determines the action of $g$ on $\tree_{\bar{d}} $. In fact, we have the following decomposition:
\[\begin{array}{ccc}
\text{Aut}\left(\tree_{\bar{d}} \right)   & \overset{\Phi}{\longrightarrow}  & \left( \text{Aut}\left(\tree_{\sigma(\bar{d})} \right) \times \dots \times \text{Aut}\left(\tree_{\sigma(\bar{d})} \right) \right)  \rtimes \mathfrak{S}_{d_0}   \\
 g  &  \longrightarrow   &  ( \varphi_{1}(g), \dots, \varphi_{d_0}(g)) . g_1 
\end{array}
\]
where $\mathfrak{S}_{d_0}$ denotes the symetric group on the set of $d_0$ elements $L_1$. Its action on $\left( \text{Aut}\left(\tree_{\sigma(\bar{d})} \right) \times \dots \times \text{Aut}\left(\tree_{\sigma(\bar{d})} \right) \right) $ is the permutation of the coordinates. The isomorphism $\Phi$ is called the \emph{recursion isomorphism}. 
 
Generalizing further, we denote by $\Phi^{(n)}$ the decomposition of $\text{Aut}\left(\tree_{\bar{d}} \right) $ with respect to the level $L_n$:
\[\begin{array}{ccc}
 \text{Aut}\left(\tree_{\bar{d}} \right)  & \overset{\Phi^{(n)}}{\longrightarrow} & \left(\prod_{w \in L_n}   \text{Aut}\left(\tree_{\sigma^n(\bar{d})} \right)  \right)  \rtimes \text{Aut}\left(\rm T_{\bar{d},n}\right)   \\
 g  &  \longrightarrow   &  \left( \varphi_w \right)_{w\in L_n}  .\ g_n 
\end{array}\]
where $\rm T_{\bar{d},n} $ is the restriction of $\tree_{\bar{d}}$ to its $n$-th first levels, and where $\text{Aut}(\rm T_{\bar{d},n} )$ is the restriction of $\text{Aut}\left(\tree_{\bar{d}} \right) $ to this stable subgraph $\rm T_{\bar{d},n} $. 

We remark that an element $g\in \text{Aut}\left(\tree_{\bar{d}} \right) $ fixes the restriction $\rm T_{\bar{d},n} $ if and only if $g_n$ equals 1, which is also equivalent to $g$ fixing the $n$-th level $L_n$.   

\section{A sufficient condition}\label{sec3}

Let $G$ be a countable group and  $\mathcal S = \left( H_k\right)_k $ a strictly decreasing sequence of finite index subgroups. To simplify the notations, let $\tree$ denote the associated rooted tree $\tree \left(\mathcal S \right)$. The goal of this section is to find a condition which ensures that $\ros$, or equivalently $\rot$, weakly contains the regular representation $\lambda$ of $G$ on $\ell^2 \left( G \right)$. This is of course only possible if the representation $\ros$ is faithful, i.e. if the action on $\tree$ is faithful.

What follows deals also with the representations $\ros^{\otimes n} \simeq \rot^{\otimes n}$ that $G$ admits on 
\[\left(\bigoplus_{k=0}^{+\infty} \ell^2 \left( G/H_k\right) \right)^{\otimes n}=\left(\ell^2 \left( \tree^0 \right) \right)^{\otimes n} = \ell^2 \left( \tree^0 \times \dots \times \tree^0\right)\] 
coming from the diagonal action of $G$ on $\tree ^n$.

\begin{theo}\label{theosuffisant} Let $G$ be a countable group acting faithfully on a rooted tree $\tree$. 
\begin{enumerate}[i.]
 \item\label{theosuffisant1} If for every vertex $v\in \tree^0$, the stabilizer $\St_G (v\tree)$ is trivial, then $\lambda \prec \rot$.
 \item \label{theosuffisant2}More generally, if the set $\bigcup_{v\in \tree^0} \St_G (v\tree)$ is finite and has cardinality $n$, then $\lambda \prec \rot^{\otimes n}$. 
 \item \label{theosuffisant3}One always has  $\lambda \prec \bigoplus_{n=1}^{+\infty}\rot^{\otimes n} $.
\end{enumerate}

\end{theo}

\begin{proof}[Proof of Theorem \ref{theosuffisant}]
i. Let us show that the action of $G$ on the countable set $ \tree^0$ is l.s.f. so that Proposition \ref{lemsuff} applies. If this is not the case, then there exists a finite subset $F$ of $G$ not containing 1 and such that $F\cap \St_G \left( v \right) $ is non-empty, for all $v \in \tree^0$. Let $\Upsilon$ be the map that associates to a vertex $v\in \tree^0$ this non-empty finite set $F\cap \St_G \left( v \right) $, and let $v_0 \in \tree^0$ be such that $\left| \Upsilon \left(v_0\right) \right| $ is minimal. It is clear that if a vertex $y$ is a descendant of a vertex $x$ (i.e. $y\in x\tree$), then $\Upsilon \left( y \right) \subset \Upsilon \left( x \right)$. By minimality of the cardinality of $\Upsilon \left( v_0 \right)$ and the previous remark, we have 
\[
 \forall v\in v_0 \tree, \ \Upsilon \left( v \right)= \Upsilon \left( v_0 \right)
\]
and therefore any element in $\Upsilon \left( v_0 \right)$, by construction necessarily non-trivial, fixes the subtree $v_0 \tree$. This conclusion contradicts the hypothesis of \ref{theosuffisant}.\ref{theosuffisant1}.
 
ii. Once again, we want to use Proposition \ref{lemsuff} so suppose that $F$ is a finite subset of $G$ not containing 1. Let $F_{free}$ be the subset of $F$ consisting of its elements which do not fix any subtree: 
\[
 F_{free} = F \setminus  \left( F \cap \bigcup_{v\in \tree^0} \St_G (v\tree) \right).
\]
We set
\[
 \left\lbrace h_1,h_2,\dots,h_k \right\rbrace    = F \setminus F_{free}.
\]
For each $h_i$, there is an element $v_i \in \tree^0$ such that 
\[\St_G (v_i) \bigcap \left(F_{free} \cup \left\lbrace h_i\right\rbrace  \right)=\emptyset.\]  
Indeed, let $w\in \tree^0$ such that $h_i(w)\neq w$: by the definitions of $w$ and $F_{free}$, no subtree of $w\tree^0$ can be fixed by an element of $F_{free} \cup \left\lbrace h_i\right\rbrace $. We thus can apply the method of \ref{theosuffisant1}. to prove the existence of such a $v_i$ in the subtree $w\tree^0$. \\

Now, we have an element $\left( v_1,\dots, v_k \right) \in \left( \tree^0\right)^k$ such that
\[
 F \cap \St_G \left( \left\lbrace \left( v_1,\dots, v_k \right) \right\rbrace \right)=F \cap \bigcap_{i=1}^{k} \St_G \left( v_i\right)=\emptyset
\]
where \[
       k= \left|F \setminus F_{free} \right| =\left| F \cap \bigcup_{v\in \tree^0} \St_G (v\tree)\right| \leq \left|\bigcup_{v\in \tree^0} \St_G (v\tree)\right|=n.
      \]

Completing the sequence $ \left( v_1,\dots, v_k \right) $ by any vertices $v_{k+1}, \dots, v_n$, we get an element in $\left( \tree^0\right)^n $ whose stabilizer in $G$ does not intersect $F$. Therefore, the action on $\left(\tree^0\right)^n$ is l.s.f. and Proposition \ref{lemsuff} concludes.   

iii. It is clear that for an action $G$ on $X$, the action of $G$ on $\bigsqcup_{n\in \N} X^n$ is l.s.f. as soon as the one on $X$ is faithful. But $\bigoplus_{n=1}^{+\infty}\rot^{\otimes n} = \rho_{\bigsqcup_{n\in \N} \left(\tree^0\right)^n}$ so that Proposition \ref{lemsuff} applies again.  
\end{proof}

\begin{remarque}[on the l.s.f. condition for actions on rooted trees]\label{rmqNormal}
 The first part of the proof shows that the triviality of all the stabilizers $\St_G \left(v\tree\right)$ implies that the action is l.s.f. If the action is spherically transitive, then the converse holds. Besides, it is easy to see that if the stabilizer of an infinite geodesic path in $\tree$ is trivial, then the action is l.s.f. In particular, if $\mathcal S = \left( H_n\right)_n $ is a decreasing sequence of finite index subgroups of $G$ with \underline{trivial intersection}, then the action on the associated rooted tree $\tree_{\mathcal S}$ is l.s.f. and the representation $\rho_{\mathcal S}$ weakly contains the regular. However, this condition is not necessary in general for the action to be l.s.f. (see for instance the realization of the \emph{lamplighter group} $\left(\oplus_{n\in \Z} \Z/2\Z\right)\rtimes \Z $ as an \emph{automaton group} in \cite{grigzukrev}). 
\end{remarque}

\begin{remarque}\label{rmqdebile}
 Notice that, for every positive integer $n$, $\ros^{\otimes n}$ is a subrepresentation of $\ros^{\otimes n+1}$. Indeed, the root of $\tree$ is fixed by $G$, therefore there are invariant vectors under $G$ in $\ell^2 \left( \tree^0 \right) $ and so the trivial representation $\epsilon$ is a subrepresentation of $\ros$. It is thus not a surprise that the sufficient condition in \ref{theosuffisant}.\ref{theosuffisant2}  is weaker than the one in \ref{theosuffisant}.\ref{theosuffisant1}. In \cite{mathese} we give, for every positive integer $n$, some concrete examples of pair $\left(G,\mathcal S\right)$ for which  $\ros^{\otimes 2^n}$ weakly contains $\lambda$, but $\ros^{\otimes 2^n -1}$ does not.  
\end{remarque}

\begin{remarque}
The statement \ref{theosuffisant}.\ref{theosuffisant3} can be compared with the well known fact that if $G$ is a \underline{finite} group and $\rho$ is a faithful representation of $G$, then there is a positive integer $N$ such that every irreducible representation of $G$ appears in $\rho^{\otimes N}$. Here we only get that the regular representation \footnote{which weakly contains every unitary representation of $G$ if and only if $G$ is amenable.} is weakly contained in the sum of all the  $\ros^{\otimes n}$. But as we will see in the last section, there are examples of pair $\left(G,\mathcal S\right)$ for which $\lambda$ is not weakly contained in $\ros^{\otimes n}$ for all $n$, and so the statement of \ref{theosuffisant}.\ref{theosuffisant3} is optimal. 
\end{remarque}

\section{A necessary condition}\label{sec4}

In this section, we study the inverse implication of Theorem \ref{theosuffisant}.\ref{theosuffisant1}. It turns out to be true only under an algebraic assumption on $G$. After the proof of Theorem \ref{theonec}, we will give an example (see \ref{exnec}) of an action of a group on a rooted tree which shows that the sufficient condition of Theorem \ref{theosuffisant}.\ref{theosuffisant1} is not necessary in general and explains the additional assumption in Theorem \ref{theonec} below.

\begin{theo}\label{theonec}
Let $G$ be a countable group in which the normalizer $N_G \left(H \right)$ of any non-central finite group $H$ has infinite index in $G$.

 Suppose that $G$ acts spherically transitively on a rooted tree $\tree$. If there exists a subtree $v\tree$ whose stabilizer $\St_G \left( v\tree\right)$ in $G$ is not trivial, then the $*$-homomorphism $\rho_\tree$ defined on $\C G$ is not faithful.   
\end{theo}

\begin{remarque}
 The conclusion of \ref{theonec} clearly implies that the representation $\rho_\tree$ does not weakly contain the regular representation $\lambda$ of $G$; indeed $\lambda$ always defines a faithful representation of $\C G$. However, these two conclusions are in general not equivalent. Indeed, the regular representation of $ \Z$ extended linearly to $\C \Z$ is, via Fourier transform, given by multiplication on $L^2 \left( S^1 \right) $ by functions $S^1 \ni z \rightarrow \sum_{i=-l}^l \alpha_i z^i$ with $\alpha_i \in \C$. If one restricts this representation to $\mathcal C \left( \mathcal A \right) $ where $\mathcal A$ is a infinite closed strict subset of $S^1$, one gets a representation which induces a faithful $*$-homomorphism of $\C \Z$ (because a non-zero function $ z \rightarrow \sum_{i=-l}^l \alpha_i z^i$ has finitely many zeros) and which cannot weakly contain the regular (because there are some functions  $ f=z \rightarrow \sum_{i=-l}^l \alpha_i z^i$ such that $\normsup{f}$ is only 
reached in $S^1 \setminus \mathcal A$). 
\end{remarque}

\begin{proof}[Proof of Theorem \ref{theonec}]
 Let $v$ be a vertex of the $n$-th level $L_n$ of $\tree$ such that $\St_G \left( v\tree\right) $ is not trivial. First note that such a group is not normal, thus non-central in $G$. Indeed, the $G$-action on the $n$-th level $L_n$ is transitive, hence 
\[\bigcap_{g\in G} g^{-1}\St_G \left( v\tree\right)g=\bigcap_{g\in G} \St_G \left( g\left(v\right) \tree\right) =\left\lbrace 1\right\rbrace.\] 

Moreover, this subgroup is normalized by the stabilizer $\St_G (L_n)$ of the $n$-th level. The later having finite index in $G$, our additional assumption on $G$ implies that $\St_G \left( v\tree\right) $ is an infinite group. Therefore, it has a non-trivial intersection with the finite index subgroup $\St_G (L_n)$. 

Now, let us define for every subset $A$ of $L_n $, the subgroup $\St_G \left( L_n ; A \tree \right) $ of $G$ of the elements which fixe $L_n$ as well as all the subtrees rooted at a vertex in $A$:
\[
\St_G \left( L_n ; A \tree \right) := \St_G \left( L_n \right) \cap \bigcap_{a \in A} \St_G \left( a \tree \right) . 
\]

The first paragraph of the proof concludes that $\St_G \left( L_n ; \left\lbrace  v\right\rbrace \tree \right) $ is not trivial. So let $A_0$ be a subset of $L_n$ such that $\St_G \left( L_n ; A_0 \tree \right) $  is non-trivial, and such that among the subsets of $L_n$ with this property, the cardinality of $A_0$ is maximal. Denote by $\left( A_0,A_1,\dots,A_N \right) $ the orbit of $A_0$ under the $G$-action. One has:
\begin{proper}
 \begin{enumerate}[(A)]
 \item\label{P1} $\bigcup_{i=0}^{N} A_i = L_n$ because $G$ acts transitively on $L_n$.
 \item\label{P2} $\forall i=1\dots N,\ \exists \alpha_i \in G $ 
$$ \St_G (L_n,A_{i} \tree)=\St_G (L_n,\alpha_i (A_{0}) \tree)=\alpha_i \St_G (L_n,A_0 \tree) \alpha_{i}^{-1} \neq \{1\}. $$ 

 \item\label{P3} For every $i \neq j$, $g_i \in \St_G \left( L_n ; A_i \tree \right)$  and $g_j \in \St_G \left( L_n ; A_j \tree \right)$, one has 
 \[
\left[ g_i,g_j\right]:=g_ig_jg_i^{-1}g_j^{-1} \in  \St_G \left( L_n ; (A_i \cup A_j) \tree \right) =\left\lbrace 1 \right\rbrace .   
 \]
Therefore every element in $\St_G \left( L_n ; A_i \tree \right) $ commutes with every element in $\St_G \left( L_n ; A_j \tree \right) $.
\end{enumerate}
\end{proper}

The properties (\ref{P1}) and (\ref{P2}) are clear. Let us check (\ref{P3}).
If we write down the decomposition of $g_i$ and $g_j$ with respect to the level $L_n$ (see Section \ref{sec2}), we get
\begin{eqnarray*}
\Phi^{(n)} \left(g_i \right)&=& \left( *,\dots,*,1,*,\dots,1,\dots \right)   \\
\Phi^{(n)} \left(g_j \right)&=& \left( *,\dots,1,1,*,\dots,*,\dots \right)
\end{eqnarray*}
where the 1s appear respectively in the positions corresponding to $ x\in A_i$ and $ y\in A_j$. Therefore
\begin{eqnarray*}
 \Phi^{(n)} \left( \left[g_i,g_j \right] \right)&=& \left( *,\dots,1,1,*,\dots,1,\dots \right)
\end{eqnarray*}
where the 1s appear in the positions corresponding to $ z\in A_i \cup A_j$. Thus, we conclude that $\left[g_i,g_j \right] $ is an element of the group 
\[ \St_G \left( L_n ; \left( A_i \cup A_j\right) \tree \right) \]
which is by construction trivial since $|A_i \cup A_j|>|A_0| $.  

Now, for all sequences $g_0,g_1,\dots,g_N$ with $g_i \in \St_G \left( L_n ; A_i \tree \right) $, we define 
\[
 M(g_0,g_1,\dots,g_N)=\prod_{i=0}^{N} \left( 1-g_i \right) \in \C G .
\]

The following two lemmas clearly imply Theorem \ref{theonec}.\end{proof}

\begin{lem}\label{lem1}
For every sequence  $g_0,g_1,\dots,g_N$ with $g_i \in \St_G \left( L_n ; A_i \tree \right) $, 
\[\rho_{\tree} \left(M(g_0,g_1,\dots,g_N) \right)=0 .\]
\end{lem}

\begin{proof}[Proof of Lemma \ref{lem1}]
As we have already seen, the decomposition of the element $g_i \in \text{Aut} \left( \tree \right)$ with respect to the $n$-th level $ L_n$ is given by
\[
 \Phi^{(n)} (g_i)= \left( *,\dots,*,1,*,\dots,1,\dots \right)
\]
where the 1s appear in the positions corresponding to $w\in E_i$. Thus, the operator $\rho_{\tree}\left( 1-g_i \right)$ restricts to 0 on the invariant subspace 
\[\left\lbrace f \in \ell^2 \left( \tree^0\right) \ | \ \text{supp}(f) \subset \bigcup_{k=0}^n L_k \cup \bigcup_{w\in E_i} w \tree\right\rbrace. \]

Because $\bigcup_{i=0}^{N} E_i = L_n $ (Property \ref{P1}), $\rho_{\tree} \left(M(g_0,g_1,\dots,g_N) \right)=0 $. 
\end{proof}

\begin{lem}\label{lem2}
There exists a sequence  $g_0,g_1,\dots,g_N$ with $g_i \in \St_G \left( L_n ; A_i \tree \right) $ such that 
\[M(g_0,g_1,\dots,g_N) \neq 0 .\]
\end{lem}

\begin{proof}[Proof of Lemma \ref{lem2}]
The group $\St_G \left( L_n ; A_i \tree \right) $ is non-central in $G$ and is normalized by the finite index subgroup $\St_G \left( L_n \right) $. Hence, it is infinite by assumption. Two cases arise:

\underline{1st case}: there exists an element $g_0 \in \St_G \left( L_n ; A_0 \tree \right)$ whose order is infinite. The groups $\St_G \left( L_n ; A_i \tree \right) $ being conjugated (Property \ref{P2}), we can choose $g_i \in \St_G \left( L_n ; A_i \tree \right)$ of infinite order. These elements commute (Property \ref{P3}) and thus generate a group of the form $\Z^l \oplus K $ where $K$ is a finite abelian group. If $m$ is the cardinality of $K$, we have
\[
 H \overset{\textrm{def}}{:=} \left\langle g_0 ^m, \dots ,g_N ^m \right\rangle = \Z^{l'}
\]
with $l' \neq 0$ because, for instance, $g_0 ^m$ has infinite order. It is well known that the algebra $\C H$ is a domain, thus $M(g_0 ^m,g_1 ^m,\dots,g_N ^m) =\prod_{i=0}^{N} \left( 1-g_i ^m \right) \neq 0 $. 

\underline{2nd case}: the groups $\St_G \left( L_n ; A_i \tree \right)$ are periodic (and infinite). Let us construct inductively a sequence $g_0,g_1,\dots,g_N $ such that:
\begin{itemize}
 \item[-] for all $i$, $g_i$ is a non trivial element of $\St_G \left( L_n ; A_i \tree \right)$, 
 \item[-] for all $i$, $g_i \notin \left\langle g_0,\dots, g_{i-1} \right\rangle.$ 
\end{itemize}

The initial step of the induction is clear. Suppose $g_0,\dots g_i$ are already constructed, with $i<N$. The group $K=\left\langle g_0 , \dots ,g_i \right\rangle $ is generated by torsion elements which commute (Property \ref{P3}): $K$ is a finite group and thus we can choose $g_{i+1}$ in the infinite group $\St_G \left( L_n ; A_{i+1} \tree \right) $ which is not in $K$.

Now $M(g_0 , \dots , g_N )=\prod_{i=0}^{N} \left( 1-g_i  \right) \neq 0 $. Indeed, expanding this product, one sees that its nullity implies a relation of the form
\[
 1=g_{i_1}g_{i_2} \dots g_{i_l} \textrm{ with } i_1<i_2< \dots <i_l \textrm{ and } l \textrm{ odd}.
\]

In particular it would follow that $g_{i_l} \in \left\langle g_{i_1},\dots, g_{i_{l-1}} \right\rangle $ which is by construction impossible.
\end{proof}

In the next example, we construct a group $G$ together with a spherically transitive action on a rooted tree $\tree $ with non-trivial stabilizers of subtrees. These stabilizers are finite, non-central and normalized by a finite index subgroup of $G$. Thus, $G$ is excluded from the framework of Theorem \ref{theonec} and in fact, we will prove that the profinite representation of $G$ on $\ell^2 \left( \tree^0 \right) $ weakly contains the regular $\lambda$.

\begin{exemple}\label{exnec}
 On the finite set $X=\left\lbrace 1,2,3,4,5,6\right\rbrace$ (which will play the role of the first level $L_1$ of $\tree$), consider the permutations 
\[
\alpha= \left(1,3,5 \right)\left(2,4,6\right), \quad \beta_r= \left(1,2 \right)\left(3,4\right).
\] 

Also let $H$ be the group generated by $\alpha$ and $\beta_r$; we denote by $\rho_1$ the permutational representation that $H$ admits on $\ell^2 \left( X \right) \simeq \C^6$ via its action on $X$. 

\begin{proper} \label{col}
\begin{enumerate}[(A)]
  \item The group $K$ generated by $\beta_r$, $\beta_l:=\alpha \beta \alpha^{-1}$ and $\beta_m:=\alpha^2 \beta \alpha^{-2}$ is a non-central normal subgroup in $H$ and 
\[
H = K \rtimes \left\langle \alpha \right\rangle = \left(\Z/2\Z \oplus\Z/2\Z \right)\rtimes \Z/3\Z. 
\]
 \item The representation $\rho_1$ extended linearly to $\C H$ defines a \underline{faithful} $*$-homomorphism into $\mathcal B \left( \ell^2 \left( X \right) \right) \simeq \text{End} \left( \C^6 \right)$. \label{cool}
\end{enumerate}
\end{proper}

\begin{proof}
(A) One has 
\[
\beta_l=\alpha \beta \alpha^{-1}= \left( 3,4\right) \left( 5,6\right) \text{ and }
\beta_m=\alpha^2 \beta \alpha^{-2}=\left( 1,2\right) \left( 5,6\right) 
\]
so that $\beta \alpha \beta \alpha^{-1}=\alpha^2 \beta \alpha^{-2}$. Therefore, the group $\left\langle\beta,\alpha \beta \alpha^{-1},\alpha^2 \beta \alpha^{-2} \right\rangle $ is $\Z/2\Z \oplus\Z/2\Z $ and is normal in $H$. Thus, 
\[
H=\left\langle\beta,\alpha \beta \alpha^{-1},\alpha^2 \beta \alpha^{-2} \right\rangle \rtimes \left\langle \alpha \right\rangle = \left(\Z/2\Z \oplus\Z/2\Z \right)\rtimes \Z/3\Z.
\]

(B) We want to prove that every irreducible representation of $H$ appears in $\rho_1$. Let us first compute the character $\tau$ of the representation $\rho_1$. Here, the map $\tau$ sends a permutation in $H$ to the cardinality of its fixed points set:

\[
 \tau \left( h\right) = \left\{ \begin{array}{ll}
6 & \text{if $h=1$}\\
2 & \text{if $h=\beta_l,\beta_m$ or $\beta_r$}\\
0 & \text{otherwise.}
\end{array} \right.
\] 

Let $\pi$ be an irreducible representation of $H$ whose character is denoted by $\psi$. We want to prove that $\left\langle\tau,\psi \right\rangle $ is non zero, i.e.
\[
  \frac{1}{|H|} \sum_{h \in H} \tau\left( h \right)\psi\left( h \right) \neq 0.
\]

Thanks to the above computation, this is equivalent to 
\[6 \psi\left( 1\right) + 2\left( \psi\left( \beta_r\right)+\psi\left( \beta_m\right)+\psi\left( \beta_l\right) \right)  \neq 0.\] 

Since the $\beta_*$ are conjugate in $G$, their image under $\pi$ have the same trace. Thus, we want to show (whatever $*$ is) that one of the following three equivalent statements is true: 
\begin{eqnarray*}
 6 \left( \psi\left( 1\right) +\psi\left( \beta_*\right) \right)\neq 0 &\Longleftrightarrow & \dim \pi +\psi\left( \beta_*\right)\neq 0 \\
                                                                 &\Longleftrightarrow & \pi\left( \beta_*\right) \neq -1.
\end{eqnarray*}
Now the last assertion is clear because $\beta_r \beta_l=\beta_m$ and it is impossible for the $\pi \left(\beta_* \right)$'s to all equal $-1$. 
\end{proof}

Next consider the action of the group $\Z$ on $\tree_{\tilde 2}:=\tree_{2,2,\dots,2,\dots}$ given by the sequence of subgroups $\mathcal S = \left(2\Z,4\Z,\dots,2^n\Z,\dots \right)$. This action is generated by a single element; let $s \in \text{Aut}\left(\tree_{\tilde 2} \right)$ be this generator. 

Before continuing our construction, we remark that Theorem \ref{theosuffisant}.\ref{theosuffisant1} and the remark \ref{rmqNormal} implies that $\rho_{\mathcal S}$ weakly contains the regular representation of $\Z$. The group $\Z$ is amenable, hence these two representations are weakly isomorphic i.e. 
\begin{equation}\label{tarace}
 \cs{\rho_{\mathcal S}}\left( \Z \right)=\cs{\rho_{\tree_{\tilde 2}}} \left( \Z \right)\simeq\cs{\lambda} \left( \Z \right). 
\end{equation}

We now consider the rooted tree $\tree_{6,2,2,\dots,2,\dots}$ whose first level $L_1$ is the set $X$ on which $H$ acts, and whose subtrees rooted at $L_1$ are all $\tree_{\tilde 2} $ on which $\Z$ acts.

\begin{defi}\label{defG}
The group $G$ is the subgroup of $\text{Aut} \left(\tree_{6,2,2,\dots,2,\dots} \right)$ generated by the elements $\bar{\alpha}$, $\bar{\beta_r} $ and $ \bar{s} $ defined via the recursion map (see Section 2):

\begin{eqnarray*}
\Phi(\bar{\alpha})&=&\left(1,1,1,1,1,1 \right) \alpha, \\
\Phi(\bar{\beta_r})&=&\left(1,1,1,1,1,1 \right) \beta_r, \\
\Phi(\bar{s})&=&\left(s,s,s,s,s,s \right).
\end{eqnarray*}
\end{defi}

The subgroup $\left\langle\bar{\alpha},\bar{\beta_r} \right\rangle$ is isomorphic to $H$ and the element $\bar{s} $ generates a copy of $\Z$ in $G$ which clearly commutes with $H$. Thus,
\begin{equation}\label{Gdecomp}
G=H \oplus \Z.
\end{equation}

The group $G$ acts spherically transitively on $\tree_{6,2,2,\dots}$ because $H$ acts transitively on its first level $L_1$, and by construction $\Z$ acts spherically transitively on all the subtrees $\tree_{\tilde 2}$ rooted at it. Moreover, the stabilizer of such a subtree is not trivial; for instance, the two right-most one (issued from the vertices labelled by $5$ and $6$) are fixed under the action of $\bar{\beta_r} $. The next proposition shows that the sufficient condition in Theorem  \ref{theosuffisant}.\ref{theosuffisant1} is not necessary and explains the additional algebraic assumption in \ref{theonec}.   

\begin{prop}\label{propex}
The representation $\rho_{\tree_{6,2,2,\dots}}$ of $G$ on $\ell^2 \left(\tree_{6,2,2,\dots}^0 \right)$ is weakly isomorphic to its regular $\lambda$. 
\end{prop}

\begin{proof}
To simplify notation, let us write $\rho$ for $\rho_{\tree_{6,2,2,\dots}}$. Maybe the easiest way to prove \ref{propex}  is to use the language of $C^*$-algebras. We want to prove that 
\[
\cs{\rho}\left( G \right)  = \cs{\lambda} \left( G \right).
\]  

We have 
\begin{eqnarray*}
\cs{\lambda}\left( G \right)&=&\cs{\lambda}\left(\Z\right)\otimes \cs{\lambda}\left(H\right)\quad \text{by (\ref{Gdecomp})}\\     
                            &=&\cs{\lambda}\left(\Z\right)\otimes \C H \quad \text{because $H$ is finite.}
\end{eqnarray*}

Thus Proposition \ref{propex} is a consequence of (\ref{tarace}) and the following lemma.
\end{proof}

\begin{lem}\label{zob}
 \[\cs{\rho}\left( G \right) = \cs{\rho_{\tree_{\bar 2}}} \left( \Z \right) \otimes \C H\]
\end{lem}
\begin{proof}[Proof of Lemma \ref{zob}]
Consider the restriction $\rho' $ of $\rho$ to the invariant subspace $\mathcal H $ of $\ell^2 \left(\tree_{6,2,2,\dots}^0 \right)$ consisting of functions null at the root of $\tree_{6,2,2,\dots} $ :
\begin{eqnarray}
  \mathcal H &=& \left\lbrace f \in \ell^2 \left(\tree_{6,2,2,\dots}^0 \right) \ |\ \text{supp} (f) \subset \bigcup_{v \in L_1} v\tree_{6,2,2,\dots} \right\rbrace\\
             &=& \bigoplus_{v \in L_1} \ell^2 \left(\tree_{\tilde 2}^0 \right).\label{merde}
\end{eqnarray}

The above restriction amounts to the removal of a copy of $\epsilon$ (the trivial representation) from $\rho$. But $\rho'$ still contains $\epsilon$ since the constant functions on a level $L_n$ belong to $\ell^2 \left( L_n \right) \subset \mathcal H$. Therefore
\[
 \cs{\rho}\left( G \right) = \cs{\rho'}\left( G \right).
\]

Now, the decomposition (\ref{merde}) of $\mathcal H $ implies that $\cs{\rho'}\left( G \right) $ is a subalgebra of $\mathcal A \otimes \text{End} \left(\C^6 \right)$ where $\mathcal A$ is the $C^*$-algebra generated by the restrictions $\varphi_v \left( g\right) \in \mathcal U \left(  \ell^2 \left(\tree_{\tilde 2}^0 \right)\right)$ with $v\in L_1$ and $g\in G$. By Definition \ref{defG} of the group $G$, the algebra $\mathcal A$ is generated by $ s$ and hence is $ \cs{\rho_{\tree_{\tilde 2}}} \left( \Z \right) $. Moreover, the only elements in $G$ acting non-trivially on $L_1$ (i.e. inducing a non-trivial permutation of the factors $\ell^2 \left(\tree_{\tilde 2}^0 \right) $ in $\mathcal H $) are the elements of the subgroup $H$. Summarizing, we have
\[
 \cs{\rho'}\left( G \right)= \cs{\rho_{\tree_{\tilde 2}}} \left( \Z \right) \otimes \rho_1 \left( \C H \right)
\]
where $\rho_1$ is the permutational representation of $H$ on $\ell^2 \left( X \right)$ defined at the beginning of this paragraph, i.e. the restriction of $\rho$ to $\ell^2 \left( L_1 \right)$. Lemma \ref{zob} is now a consequence of Property \ref{col}.\ref{cool}. 
\end{proof}

\end{exemple}

\section{Examples and applications}\label{sec5}

In this last section, we want to illustrate the results of the previous two. In particular, we will see that the sufficient condition of \ref{theosuffisant}.\ref{theosuffisant1} is automatically fullfilled for many lattices in semisimple real Lie groups, independently on the choice of a sequence $\mathcal S$ defining a faithful representation $\ros$. We will also see that for weakly branched groups, the situation is diametrically different. 

\subsection{Lattices in semisimple real Lie groups, hyperbolic groups}\label{subsecLattices}

\subsubsection{Higher $\R$-rank case.}
The following proposition is a direct consequence of the Margulis Normal Subgroup Theorem \cite{margulis78,margulis79,zimmer81}. 

\begin{prop}\label{propHR}
 Let $G$ be a connected semisimple real Lie group with finite centre, no compact factor and $\R$-rank $\geq 2$. Let $\Gamma$ be an irreducible lattice in $G$ acting faithfully and spherically transitively on a rooted tree $\tree$. Then for all subtrees $v\tree$ of $\tree$, the stabilizer $\St_{\Gamma} \left( v\tree \right)$ is trivial. 
\end{prop}

\begin{proof}
Let $v$ be a vertex of $\tree$. The group $\St_{\Gamma} \left(v \right)$ has finite index in $\Gamma$, hence is also an irreducible lattice in $G$. The stabilizer $\St_{\Gamma}\left( v\tree \right)$ of the subtree rooted at $v$ is a normal subgroup of $\St_{\Gamma} \left( v \right)$. Moreover, $\St_{\Gamma}\left( v\tree \right)$ has infinite index in $\Gamma$ (because $\Gamma$ acts spherically transitively on the subtree $v\tree $). Therefore it has infinite index in  $\St_{\Gamma} \left( v \right)$. By the Margulis Normal Subgroup Theorem, $\St_{\Gamma}\left( v\tree \right)$ is central in $G$ and thus in $\Gamma$. From this follows that $\St_{\Gamma}\left( v\tree \right)$ is trivial, because $\Gamma$ acts transitively on $L_n$ so that 
\[
\bigcap_{\gamma \in \Gamma} \gamma^{-1}\St_{\Gamma} \left( v\tree\right)\gamma=\left\lbrace 1\right\rbrace. 
\]
\end{proof}

\subsubsection{Hyperbolic groups and $\R$-rank=1 case.}\label{subsecHyp}

The first result of this paragraph deals with Gromov hyperbolic groups. We refer to \cite{gromovhyp87,ghysharpehyp} for details and proofs of the general facts which we use in the proof below. 

\begin{prop}\label{prophyp}
Let $\Gamma$ be a Gromov hyperbolic group. Assume that $\Gamma$ is non-elementary and acts faithfully and spherically transitively on a rooted tree $\tree$. Then $\bigcup_{v\in \tree^0} \St_{\Gamma} (v\tree)$ is finite. 
\end{prop}

\begin{remarque}
 This result is optimal. Indeed, in Example \ref{exnec}, replace the 6 subtrees rooted at the first level on which $\Z$ acts diagonally by 6 rooted trees $\tree_{\bar d}$ on which a residually finite hyperbolic group $\Gamma$ acts faihfully and spherically transitively. The subgroup of $\text{Aut}\left( \mathcal T _{6,\bar d}\right)$ that we get is then $H \times \Gamma$ and is hyperbolic because $H$ is finite. Here again, by construction, its action on $\mathcal T _{6,\bar d} $ is faithful and spherically transitive, but the subgroup $H$ fixes the subtrees rooted at the two right-most vertices of the first level.  
\end{remarque}

\begin{proof}
First, we prove that for every vertex $v\in \tree$, the stabilizer $\St_{\Gamma} (v\tree)$ is finite. If this is not the case, being a subgroup of a hyperbolic group,  $\St_{\Gamma} (v\tree)$ contains an element of infinite order $\gamma$. Let $n$ be the level of $v$, i.e. $v \in L_n$; replacing $\gamma$ by the non trivial element $\gamma^{N !}$ where $N=\left| L_n \right|$, we can assume also that $\gamma$ fixes the $n$-th level $L_n$. Once again, we consider for every subset $\mathcal A$ of $L_n$
\[
 \St_{\Gamma} \left( L_n ; A \tree \right) := \St_{\Gamma} \left( L_n \right) \cap \bigcap_{a \in A} \St_{\Gamma} \left( a \tree \right) . 
\]

Above, we showed that $\St_{\Gamma} \left( L_n ; \left\lbrace v \right\rbrace \tree \right)$ is not trivial, and even contains a infinite order element. So consider $\mathcal A_{max}$ a subset of maximal cardinality for which $\St_{\Gamma} \left( L_n ; \mathcal A_{max} \tree \right)$ contains an infinite order element, say $\gamma$. Of course, $\mathcal A_{max}$ is a strict subset of $L_n$. Hence, let us choose $v \in \mathcal A_{max} $, $w \in L_n \setminus \mathcal A_{max}$, and $\sigma \in \Gamma$ that such $\sigma (v)=w $ (this is possible because the action of $\Gamma$ on $L_n$ is assumed transitive). 

\begin{lem}\label{wtf}
There exists $h\in \St_{\Gamma} \left( L_n \right)$ and integer $k$ such that 
\[
\left[\gamma^k, h\sigma\gamma^k \sigma^{-1} h^{-1} \right]:= \gamma^k h \sigma \gamma^k  \sigma^{-1} h^{-1} \gamma^{-k} h \sigma\gamma^{-k}  \sigma^{-1} h^{-1}   
\]
has infinite order. 
\end{lem}

\begin{proof}
It is sufficient to prove that there is an $h \in \St_{\Gamma} \left( L_n \right)$ such that  
\[
 \text{Fix}_{\partial \Gamma} \left(\gamma \right)\cap \text{Fix}_{\partial \Gamma} \left(h\sigma\gamma\sigma^{-1} h^{-1}\right)=\emptyset
\]
where $\text{Fix}_{\partial \Gamma} \left(\gamma \right)=\left\lbrace \gamma^{+\infty},\gamma^{-\infty}  \right\rbrace $ and $\text{Fix}_{\partial \Gamma} \left( h\sigma\gamma\sigma^{-1} h^{-1} \right)=\left\lbrace h\sigma (\gamma^{+\infty}),h\sigma (\gamma^{-\infty})  \right\rbrace $ are the pairs of fixed points of $\gamma $  and $ h\sigma\gamma\sigma^{-1} h^{-1} $ in the boundary $\partial \Gamma$ of $\Gamma$. Indeed, it is known that in that case, the group generated by $\gamma^k$ and $ \left(h\sigma\gamma\sigma^{-1} h^{-1}\right)^k =h\sigma\gamma^k\sigma^{-1} h^{-1} $ is the free group $\mathbb F _2$, as soon as $k$ is big enough. 

If $\text{Fix}_{\partial \Gamma} \left(\gamma \right)\cap \text{Fix}_{\partial \Gamma} \left(\sigma\gamma\sigma^{-1}\right)=\emptyset$, we take $h=1$. If this is not the case, we want to prove that $\St_{\Gamma} \left( L_n \right)$ is not included in the subgroup $P$ of $G$ consisted of the elements preserving $\text{Fix}_{\partial \Gamma} \left(\gamma \right) $. This is easy because this last group contains $\left\langle \gamma \right\rangle $ as a finite index subgroup, whereas $\St_{\Gamma} \left( L_n \right)$ has finite index in $\Gamma$: as $\Gamma$ is non-elementary, $\St_{\Gamma} \left( L_n \right)$ is not amenable and thus cannot be a subgroup of $P$ (which is quasi-isometric to $\Z$).
\end{proof}

Fix $h$ and $k$ like in the statement of the previous lemma. Since $\gamma^k $ belongs to $ \St_{\Gamma} \left( L_n ; \mathcal A_{max} \tree \right)$ and $h$ is in $\St _{\Gamma} ( L_n)$, the element $h\sigma\gamma^k \sigma^{-1} h^{-1}$ belongs to the group $ \St_{\Gamma} \left( L_n ; \sigma(\mathcal A_{max}) \tree \right)$. If we write down the decomposition of these two elements with respect to the level $L_n$, we get
\begin{eqnarray*}
\Phi^{(n)} \left(\gamma^k \right)&=& \left( *,\dots,*,1,*,\dots,1,\dots \right)   \\
\Phi^{(n)} \left( h\sigma\gamma^k \sigma^{-1} h^{-1}\right)&=& \left( *,\dots,1,1,*,\dots,*,\dots \right)
\end{eqnarray*}
where the 1s appear respectively in the positions corresponding to $ x\in  \mathcal A_{max}$ and $ y\in  \sigma\left(\mathcal A_{max}\right)$. Therefore
\begin{eqnarray*}
 \Phi^{(n)} \left( \left[\gamma^k, h\sigma\gamma^k \sigma^{-1} h^{-1} \right]  \right)&=& \left( *,\dots,1,1,*,\dots,1,\dots \right)
\end{eqnarray*}
where the 1s appear in the positions corresponding to $ z\in \mathcal A_{max} \cup  \sigma(\mathcal A_{max})$. Lemma \ref{wtf} then implies that $\left[\gamma^k, h\sigma\gamma^k \sigma^{-1} h^{-1} \right] $ is an infinite order element of the group 
\[ \St_{\Gamma} \left( L_n ;\mathcal B \tree \right) \]
with $\mathcal B =\mathcal A_{max} \cup \sigma(\mathcal A_{max})$. This contradicts the maximality of $\left|\mathcal A_{max} \right| $ because $w \in \sigma(\mathcal A_{max})\setminus \mathcal A_{max}$ and so $\left|\mathcal A_{max} \right|<\left|\mathcal B \right| $. We just proved that all the stabilizers $\St_G\left(v\tree\right)$ are finite. \\

To complete the proof of \ref{prophyp}, recall that in a hyperbolic group, there are only finitely many conjugacy classes of finite group. Hence, there exists a non-negative integer $N$ such that 
\begin{eqnarray*}
&\forall k \in \N, \ \forall v \in L_{N+k},\  \exists g\in G,\  i\leq N \text{ and } w \in L_i \text{ such that }  & \\
&\St_{\Gamma} (v\tree) = g  \St_{\Gamma} (w\tree)  g^{-1}.&
\end{eqnarray*}
As $ g  \St_{\Gamma} (w\tree)  g^{-1} = \St_{\Gamma} \left( g(w) \tree \right) $, we conclude that 
\[
\bigcup_{v\in \tree^0} \St_{\Gamma} (v\tree) = \bigcup_{i=1}^{N} \bigcup_{v \in L_i} \St_{\Gamma} (v\tree) 
\]
is finite.
\end{proof}

As a direct consequence of this proposition and Theorem \ref{theosuffisant}.\ref{theosuffisant2}, we obtain:

\begin{cor}
 Let $\Gamma$ be non-elementary, residually finite hyperbolic group. Let $\mathcal S$ be a decreasing sequence of finite index subgroups of $\Gamma$ such that the representation $\ros$ is faithful. Then, there exists a positive integer $n$ such that the representation $\ros^{\otimes n}$ of $\Gamma$ weakly contains the regular $\lambda$.
\end{cor}

The next application shows how to use Proposition \ref{prophyp} to get information on the spectrum of Schreier graphs.

\begin{cor}\label{corHTS}
Let $\Gamma$ be non-elementary, torsion free, residually finite hyperbolic group.  Let $\mathcal S = \left( H_n\right)_n $ be a decreasing sequence of finite index subgroups of $\Gamma$ such that the representation $\ros$ is faithful. 

Let $M_F := \frac{1}{2|F|}\sum_{g\in F} g+g^{-1}$ be the Markov operator associated to a finite generating set $F$ of $\Gamma$ which does not contain 1. Then 
\[
 \left[ -\frac{1}{|F|}, \frac{\sqrt{2|F|-1}}{|F|}\right] \subset \overline{\bigcup_{n} \text{sp} \left( \lambda_{\Gamma/H_n}\left(M_F \right) \right)} .
\]
   \end{cor}
\begin{proof}
As $M$ is self-adjoint, the closure of $\bigcup_{n} \text{sp} \left( \lambda_{\Gamma/H_n}\left(M_F \right) \right) $ equals the spectrum of $\ros \left( M \right)$ and is a compact subset of $\R$. Now, the assumption that $\Gamma$ is torsion free, Proposition \ref{prophyp} and Theorem \ref{theosuffisant}.\ref{theosuffisant1} together imply that $\ros$ weakly contains the regular representation $\lambda$ of $\Gamma$. In particular,
\[
 \text{sp}\left( \lambda \left( M_F\right)\right) \subset  \text{sp} \left( \ros\left(M_F \right) \right)=\overline{\bigcup_{n} \text{sp} \left( \lambda_{\Gamma/H_n}\left(M_F \right) \right)}.
\]

Thanks to Proposition 5 in \cite{HPRGV93}, we know that $\text{sp}\left( \lambda \left( M\right)\right)$ contains a pair $\left\lbrace m, M  \right\rbrace $ with $m\leq -\frac{1}{|F|} $ and $M \geq \frac{\sqrt{2|F|-1}}{|F|}$. Since the Baum–Connes conjecture is true for hyperbolic groups \cite{lafforgue02,mineyev-guoliang02} and $\Gamma$ is torsion free, it fullfills the Kadison-Kaplansky conjecture. Therefore, the spectrum $\text{sp}\left( \lambda \left( M\right)\right) $ is connected. This proves Corollary \ref{corHTS}.
\end{proof}

We conclude with an analogous result to Proposition \ref{propHR}, in the case of uniform lattices in $\R$-rank 1 Lie groups.  

\begin{cor}
Let $\Gamma$ be a uniform lattice in a connected simple real Lie group $G$ with finite center and $\R$-rank 1. Assume that $\Gamma$ acts faithfully and spherically transitively on a rooted tree $\tree$. Then for all subtrees $v\tree$ of $\tree$, the stabilizer $\St_{\Gamma} \left( v\tree \right)$ is trivial and thus the representation $\rot$ of $\Gamma$ weakly contains the regular $\lambda$.
\end{cor}
\begin{proof}
Such a uniform lattice is hyperbolic. Let $v$ be a vertex of $\tree$. Proposition \ref{prophyp} implies that the stabilizer $\St_{\Gamma} \left( v\tree \right) $ is a finite group. Therefore, its normalizer $N_G \left(\St_{\Gamma} \left( v\tree \right) \right)$ in $G$  is Zariski closed. Moreover, $\St_{\Gamma} \left( v\tree \right)$ is a normal subgroup of $\St_{\Gamma} \left( v \right)$; the later, having finite index in $\Gamma$, is also a lattice in $G$. By the Borel Density Theorem, the Zariski closure of $\St_{\Gamma} \left( v \right)$ is $G$ and therefore $\St_{\Gamma} \left( v\tree \right) $ is a finite normal subgroup of $G$. Hence, $\St_{\Gamma} \left( v\tree \right) $ has to be central in $G$, thus in $\Gamma$. Finally,
\[
 \St_{\Gamma} \left( v\tree \right) = \bigcap_{\gamma \in \Gamma} \gamma\St_{\Gamma} \left( v\tree \right)\gamma^{-1}=\left\lbrace 1 \right\rbrace, 
\]
the last equality coming from the transitivity of the $\Gamma$-action on $L_n$. Theorem \ref{theosuffisant}.\ref{theosuffisant1} applies.
\end{proof}

\subsection{Weakly branched groups}
This last section deals with natural examples of couple $\left(G,\mathcal T\right) $ where $G$ is a finitely generated group acting faithfully and spherically transitively on a rooted tree $\tree$ for which the conditions in Theorem \ref{theosuffisant}.\ref{theosuffisant1} and \ref{theosuffisant}.\ref{theosuffisant2} are ''far'' from being true. These examples come from the class of \emph{weakly branched} groups.

\begin{defi}
 Let $\tree$ be a \emph{regular} rooted tree, which means that $\tree = \tree_{\tilde d}$ where $\tilde d$ is the constant sequence $\tilde d=d,d,\dots,d,\dots$ with $d>1$. A finitely generated subgroup $G$ of $\aut$ is said weakly branched if its action on $\tree$ is spherically transitive and 
\[
 \forall v\in \tree^0, \ \Ri_G \left(v\right) := \bigcap_{w\in L_n \setminus \left\lbrace v\right\rbrace } \St_G \left( w\tree\right)
\]
is non-trivial. 
\end{defi}
The subgroup $\Ri_G \left(v\right) $ consists of elements which only act on the subtree $v\tree$. It is called the \emph{rigid stabilizer} of $v$. It is easy to see that if $G$ is weakly branched, these rigid stabilizers are infinite. So \emph{a fortiori} are the stabilizers $\St_G \left(v\tree\right)$ of subtrees. 

We refer to \cite{bourbakiandrzej,surveybranche} for a survey on weakly branched groups. Concerning the link between $\rot$ and $\lambda$ for such a weakly branched group, one has:

\begin{prop}\label{propwb}
 Let $G$ be a weakly branched subgroup of $\aut$. For every $n>0$, the $*$-homomorphism $\rot^{\otimes n}$ defined on $\C G$ is not faithful. In particular, the representation $\rot^{\otimes n}$ does not weakly contain the regular $\lambda$.
\end{prop}

\begin{proof}
The method is the same that the one we used for the proof of Theorem \ref{theonec}. First, we see that remark \ref{rmqdebile} implies that we need only to prove the proposition for integers of the form $d^n -1$. For every vertex $v$ in the $n$-th level $L_n$ of $\tree$, choose $g_v$ a non-trivial element in the infinite group $\Ri_G \left(w\right)$. We have
\begin{equation}\label{lasteq}
  \Phi^{(n)}\left(g_v \right) = \left(1,1,\dots,1,\varphi_v \left(g_v\right),1,\dots \right)
\end{equation}
where the non-trivial element $\varphi_v \left( g_v \right)$ appears in the position corresponding to $v$. It is clear that the $g_v$'s commute. Let  
\begin{equation*}
M=\prod_{v\in L_n} \left( 1- g_v \right) \in \C G. 
\end{equation*}

Then, $M \neq 0$. Indeed, the nullity of $M$ would imply the existence of a subset $\mathcal A$ of $L_n$ (of odd cardinality) such that 
\[1=\prod_{v\in \mathcal A} g_v,\]
and this is impossible because, if $w$ is any vertex in $\mathcal A$, (\ref{lasteq}) implies that 
\[\varphi_w \left(\prod_{v\in \mathcal A}  g_v \right)=\varphi_w \left(g_w \right)\neq 1.\]

Let us show that $\rot^{\otimes d^n -1} \left( M \right)$ is 0. This is equivalent to proving that for every $n$-tuple $\left(z_1,\dots,z_{d^n -1} \right)$ consisting of elements in $\tree^0$, 

\[\rot^{\otimes d^n -1} \left( M \right) \left(\delta_{z_1}\otimes \dots \otimes\delta_{z_{d^n -1}} \right)=0 \]

There is necessarily a vertex $v_0$ among the $d^n$ in $L_n$ such that the subtree $v_0\tree$ does not contain any $z_i$. By construction, 
\[\rot^{\otimes d^n -1} \left(1-g_{v_0} \right)\left(\delta_{z_1}\otimes \dots \otimes\delta_{z_{d^n -1}} \right)=0\]

As the $g_v$'s commute, 
\[
 M=\left( \prod_{v \in L_n \setminus \left\lbrace v_0\right\rbrace } \left( 1-g_v\right) \right) (1-g_{v_0})
\]
and this implies $\rot^{\otimes d^n -1}\left( M \right) \left(\delta_{z_1}\otimes \dots \otimes\delta_{z_{d^n -1}} \right) =0 $. 
\end{proof}

\bibliographystyle{amsalpha}
\bibliography{bibliothese}
\end{document}

%% file: macrosmath.tex
\newcommand{\N}{\mathbb{N}}
\newcommand{\Z}{\mathbb{Z}}

\newcommand{\R}{\mathbb{R}}
\newcommand{\C}{\mathbb{C}}

\newcommand{\St}{\text{Stab}}
\newcommand{\Ri}{\text{Rist}}

\newcommand{\tree}{\mathcal{T}}
\newcommand{\aut}{\text{Aut}\left(\mathcal{T} \right)}
\newcommand{\ros}{\rho_{\mathcal{S}}}
\newcommand{\rot}{\rho_{\mathcal{T}}}

\providecommand{\cs}[1]{C^{*}_{#1}}

\providecommand{\normsup}[1]{\lVert#1\rVert_{_{\infty}}}


%% file: Article-new.bbl
\def\cprime{$'$} \def\cprime{$'$} \def\cprime{$'$} \def\cprime{$'$}
\providecommand{\bysame}{\leavevmode\hbox to3em{\hrulefill}\thinspace}
\providecommand{\MR}{\relax\ifhmode\unskip\space\fi MR }
\providecommand{\MRhref}[2]{%
  \href{http://www.ams.org/mathscinet-getitem?mr=#1}{#2}
}
\providecommand{\href}[2]{#2}
\begin{thebibliography}{dlHRV93}

\bibitem[Ale72]{aleshin72}
S.~V. Ale{\v{s}}in, \emph{Finite automata and the {B}urnside problem for
  periodic groups}, Mat. Zametki \textbf{11} (1972), 319--328. \MR{MR0301107
  (46 \#265)}

\bibitem[Bar03]{bartnonunif}
L.~Bartholdi, \emph{A {W}ilson group of non-uniformly exponential growth}, C.
  R. Math. Acad. Sci. Paris \textbf{336} (2003), no.~7, 549--554. \MR{MR1981466
  (2004c:20051)}

\bibitem[BG{\v{S}}03]{surveybranche}
L.~Bartholdi, R.~I. Grigorchuk, and Z.~{\v{S}}uni{\'k}, \emph{Branch groups},
  Handbook of algebra, {V}ol. 3, North-Holland, Amsterdam, 2003, pp.~989--1112.
  \MR{MR2035113 (2005f:20046)}

\bibitem[BV05]{moybasilica}
L.~Bartholdi and B.~Vir{\'a}g, \emph{Amenability via random walks}, Duke Math.
  J. \textbf{130} (2005), no.~1, 39--56. \MR{MR2176547 (2006h:43001)}

\bibitem[Coo93]{coo93}
M.~Coornaert, \emph{Mesures de {P}atterson-{S}ullivan sur le bord d'un espace
  hyperbolique au sens de {G}romov}, Pacific J. Math. \textbf{159} (1993),
  no.~2, 241--270. \MR{1214072 (94m:57075)}

\bibitem[Dix77]{Dix77}
J.~Dixmier, \emph{{$C\sp*$}-algebras}, North-Holland Publishing Co., Amsterdam,
  1977, Translated from the French by Francis Jellett, North-Holland
  Mathematical Library, Vol. 15. \MR{0458185 (56 \#16388)}

\bibitem[dlHRV93]{HPRGV93}
P.~de~la Harpe, A.~G. Robertson, and A.~Valette, \emph{On the spectrum of the
  sum of generators for a finitely generated group}, Israel J. Math.
  \textbf{81} (1993), no.~1-2, 65--96. \MR{1231179 (94j:22007)}

\bibitem[Fel60]{Fell60}
J.~M.~G. Fell, \emph{The dual spaces of {$C^{\ast} $}-algebras}, Trans. Amer.
  Math. Soc. \textbf{94} (1960), 365--403. \MR{0146681 (26 \#4201)}

\bibitem[Fel63]{Fell63}
\bysame, \emph{Weak containment and {K}ronecker products of group
  representations}, Pacific J. Math. \textbf{13} (1963), 503--510. \MR{0155932
  (27 \#5865)}

\bibitem[GdlH90]{ghysharpehyp}
{\'E}.~Ghys and P.~de~la Harpe (eds.), \emph{Sur les groupes hyperboliques
  d'apr\`es {M}ikhael {G}romov}, Progress in Mathematics, vol.~83, Birkh\"auser
  Boston Inc., Boston, MA, 1990, Papers from the Swiss Seminar on Hyperbolic
  Groups held in Bern, 1988. \MR{MR1086648 (92f:53050)}

\bibitem[Gri84]{grig84}
R.~I. Grigorchuk, \emph{Degrees of growth of finitely generated groups and the
  theory of invariant means}, Izv. Akad. Nauk SSSR Ser. Mat. \textbf{48}
  (1984), no.~5, 939--985. \MR{MR764305 (86h:20041)}

\bibitem[Gro87]{gromovhyp87}
M.~Gromov, \emph{Hyperbolic groups}, Essays in group theory, Math. Sci. Res.
  Inst. Publ., vol.~8, Springer, New York, 1987, pp.~75--263. \MR{MR919829
  (89e:20070)}

\bibitem[G{\.Z}01]{grigzukrev}
R.~I. Grigorchuk and A.~{\.Z}uk, \emph{The lamplighter group as a group
  generated by a 2-state automaton, and its spectrum}, Geom. Dedicata
  \textbf{87} (2001), no.~1-3, 209--244. \MR{MR1866850 (2002j:60009)}

\bibitem[G{\.Z}02]{grigzukbasilica}
\bysame, \emph{On a torsion-free weakly branch group defined by a three state
  automaton}, Internat. J. Algebra Comput. \textbf{12} (2002), no.~1-2,
  223--246, International Conference on Geometric and Combinatorial Methods in
  Group Theory and Semigroup Theory (Lincoln, NE, 2000). \MR{MR1902367
  (2003c:20048)}

\bibitem[Laf02]{lafforgue02}
V.~Lafforgue, \emph{{$K$}-th\'eorie bivariante pour les alg\`ebres de {B}anach
  et conjecture de {B}aum-{C}onnes}, Invent. Math. \textbf{149} (2002), no.~1,
  1--95. \MR{1914617 (2003d:19008)}

\bibitem[Mar78]{margulis78}
G.~A. Margulis, \emph{Factor groups of discrete subgroups and measure theory},
  Funktsional. Anal. i Prilozhen. \textbf{12} (1978), no.~4, 64--76. \MR{515630
  (80k:22005)}

\bibitem[Mar79]{margulis79}
\bysame, \emph{Finiteness of quotient groups of discrete subgroups},
  Funktsional. Anal. i Prilozhen. \textbf{13} (1979), no.~3, 28--39. \MR{545365
  (80k:22006)}

\bibitem[MY02]{mineyev-guoliang02}
I.~Mineyev and G.~Yu, \emph{The {B}aum-{C}onnes conjecture for hyperbolic
  groups}, Invent. Math. \textbf{149} (2002), no.~1, 97--122. \MR{1914618
  (2003f:20072)}

\bibitem[Nek05]{livrenek}
V.~Nekrashevych, \emph{Self-similar groups}, Mathematical Surveys and
  Monographs, vol. 117, American Mathematical Society, Providence, RI, 2005.
  \MR{MR2162164 (2006e:20047)}

\bibitem[Nek09]{nekcstar}
\bysame, \emph{{$C^*$}-algebras and self-similar groups}, J. Reine Angew. Math.
  \textbf{630} (2009), 59--123. \MR{MR2526786}

\bibitem[Pla10]{mathese}
J.-F. Planchat, \emph{Groupes d'automates et $c^*$-algèbres}, Ph.D. thesis,
  Université Paris 7 Denis Diderot, 2010.

\bibitem[Wil04a]{wilsonbis}
J.~S. Wilson, \emph{Further groups that do not have uniformly exponential
  growth}, J. Algebra \textbf{279} (2004), no.~1, 292--301. \MR{MR2078400
  (2005e:20066)}

\bibitem[Wil04b]{wilson}
\bysame, \emph{On exponential growth and uniformly exponential growth for
  groups}, Invent. Math. \textbf{155} (2004), no.~2, 287--303. \MR{MR2031429
  (2004k:20085)}

\bibitem[Zim84]{zimmer81}
R.~J. Zimmer, \emph{Ergodic theory and semisimple groups}, Monographs in
  Mathematics, vol.~81, Birkh\"auser Verlag, Basel, 1984. \MR{776417
  (86j:22014)}

\bibitem[{\.Z}uk08]{bourbakiandrzej}
A.~{\.Z}uk, \emph{Groupes engendr\'es par les automates}, Ast\'erisque (2008),
  no.~317, Exp. No. 971, viii, 141--174, S{\'e}minaire Bourbaki. Vol.
  2006/2007. \MR{MR2487733}

\end{thebibliography}
